\documentclass{article}
\usepackage[utf8]{inputenc}
\usepackage[letterpaper, portrait, margin=1in]{geometry}
\usepackage{url}
\usepackage{amsmath}
\usepackage{amsthm}
\usepackage{amssymb}
\usepackage{asymptote}
\usepackage{amsfonts}
\usepackage{enumitem}
\newtheorem{theorem}{Theorem}[section]
\newtheorem{lemma}[theorem]{Lemma}
\theoremstyle{definition}
\newtheorem{definition}[theorem]{Definition}
\newtheoremstyle{example}
{\topsep}
{\topsep}
{}
{}
{\bfseries}
{.}
{.5em} 
{} 
\theoremstyle{example}
\newtheorem{example}[theorem]{Example}
\newtheoremstyle{notation}
{\topsep}
{\topsep}
{}
{}
{\bfseries}
{.}
{.5em} 
{} 
\theoremstyle{notation}
\newtheorem{notation}[theorem]{Notation}
\title{\vspace{-5ex} \huge NIM with Cash: A Concrete Appraoch}
\date{\vspace{-10ex}}
\author{
  Douglas Chen \\
  Mount Hebron High School
  \and
  Dr. William Gasarch\\
  University of Maryland College Park
}
\begin{document}
\maketitle

\normalsize
\vspace*{0.5in}
\begin{abstract}
    Let $A$ be a finite subset of $\mathbb{N},$ and let $n\in\mathbb{N}.$ then $\text{NIM}(A;n)$ is the two player game in which players alternate removing $a\in A$ stones from a pile with $n$ stones; the first player who cannot move loses. This game has been researched thoroughly. 
    \par We discuss a variant of NIM in which Player 1 and Player 2 start with $d$ and $e$ dollars, respectively. When a player removes $a\in A$ stones from the pile, he loses $a$ dollars. The first player who cannot move loses, but this can now happen for two reasons:
    \begin{enumerate}
        \item The number of stones remaining is less than $\min(A).$
        \item The player has less than $\min(A)$ dollars.
    \end{enumerate}
    This game leads to much more interesting win conditions than regular NIM.
    \par We investigate general properties of this game. We then  obtain and prove win conditions for the sets $A=\{1,L\}$ and $A=\{1,L,L+1\}.$
\end{abstract}
\section{Introduction}

\begin{notation}We use $A \subseteq^{\mathrm{fin}} B$ to indicate that $A$ is a finite subset of $B.$ \end{notation}
\par\noindent
\begin{definition}Let $A \subseteq^{\mathrm{fin}} \mathbb{N}$ and let $n \in\mathbb{N}.$ NIM$(A;n)$ is played as follows:
\begin{itemize}
    \item There is a pile of $n$ stones, with two players, Player 1 and Player 2, alternating moves. Player 1 goes first. A player's move consists of removing $a \in A$ stones from the pile.
    \item A player loses if they cannot move. If there are less than $\min(A)$ stones remaining, then the player loses.
\end{itemize}
\end{definition}
\par\noindent
\begin{notation} When we say \textit{a player takes} $s,$ we mean that the player takes $s$ stones from the pile. \end{notation}
\begin{definition} Player 1 \textit{wins} when the player has a strategy that will make them win the game no matter what Player 2 does; similar for Player 2 \textit{wins.} If the game uses the set $A \subseteq^{\mathrm{fin}} \mathbb{N},$ we use $W_{A}(n)=1$ to denote that Player 1 wins when there are $n$ stones in the pile; similar for Player 2. With $A$ understood, we can simply use $W(n).$\end{definition}

NIM is an example of a {\it combinatorial game}. Such games
have a vast literature
(see the selected bibliography of Frankel~\cite{combgamesbib}).
Variants on the 1-pile version have included
letting the number of stones a player can remove
depend on how many stones are in the pile~\cite{NIMarb},
letting the number of stones a player can remove depend on the player~\cite{partizan},
allowing three players~\cite{threenim},
viewing the stones as cookies that may spoil~\cite{cookie},
and others.
Grundy~\cite{grundy} and Sprague~\cite{sprag} showed how to analyze
many-pile NIM games by analyzing the 1-pile NIM games that it consists of.
NIM games are appealing because they are easy to explain,
yet involve interesting (and sometimes difficult) mathematics
to analyze.

\begin{example}\label{ex:NIMwinconditions} Below are the win conditions for NIM games on the sets we are interested in: $A=\{1,L\}$ and $A=\{1,L,L+1\}.$
\begin{itemize}
    \item Let $L$ be even and $\geq 2,$ with $A=\{1,L\}.$ then $W(n)=2$ iff $n\equiv 0,2,4,\cdots,L-2\pmod{L+1}.$
    \item Let $L$ be even and $\geq 2,$ with $A=\{1,L,L+1\}.$ then $W(n)=2$ iff $n\equiv 0,2,4,\cdots,L-2\pmod{2L}.$
    \item Let $L$ be odd and $\geq 3,$ with $A=\{1,L,L+1\}.$ then $W(n)=2$ iff $n\equiv 0,2,4,\cdots,L-1\pmod{2L+1}.$
\end{itemize}
\end{example}
\begin{definition} Let $A \subseteq^{\mathrm{fin}} \mathbb{N},$ let $n\in\mathbb{N},$ and let $d,e \in \mathbb{N}\cup \infty$ ($d=\infty$ means that Player $1$ has infinite cash, similar for Player $2$). The \textit{NIM with Cash} game $\text{NIM}(A;n;d,e)$ is played as follows:
\begin{itemize}
    \item Two players, Player 1 and Player 2, alternate moves. Player 1 goes first. A player's move consists of removing $a \in A$ stones from the pile, and losing $a$ dollars.
    \item Before any playing, there are $n$ stones in the pile, with Player 1 having $d$ dollars and Player 2 having $e$ dollars.
    \item A player loses if they cannot move. Either, there are less than $\min(A)$ stones in the pile, or the player does not have enough money to take away $\min(A)$ stones. 
\end{itemize}
\end{definition}
\par\noindent
\begin{notation} Player 1 \textit{wins} when the player has a strategy that will make them win the game no matter what Player 2 does; similar for Player 2 \textit{wins.} If the game uses the set $A \subseteq^{\mathrm{fin}} \mathbb{N},$ we use $W_{A}^{\text{cash}}(n;d,e)=1$ to denote that Player 1 wins when there are $n$ stones in the pile, the player has $d$ dollars, and Player 2 has $e$ dollars; similar for Player 2. With the game understood, we will simply refer to this as $W(n;d,e).$
\end{notation}
\begin{notation}Let $A \subseteq^{\mathrm{fin}} \mathbb{N}.$ NIM$(A)$ denotes the NIM with Cash game using the set $A.$
\end{notation}
\begin{definition}\ \ \ Let $A \subseteq^{\mathrm{fin}} \mathbb{N}.$ For NIM$(A):$
\begin{itemize}
    \item We say that Player 1 wins \textit{normally} if $W(n;d,e)=1,$ and the player has enough money to employ the strategy used in regular NIM. Note that this is equivalent to $W(n;d,\infty)=1.$ Similar for Player 2.
    \item If a player wins by removing $\min(A)$ on every move, we say that the player wins \textit{miserly.}
\end{itemize}
\end{definition}
\par We are interested in determining who wins for all cases for NIM$\{1,L\}$ and NIM$\{1,L,L+1\}.$ We wrote a dynamic program that will, on input $A,n,d,e,$ determine which player wins. This program runs in time $O(n)$. In addition, this method lacks the simplicity of just a formula, as illustrated in Example \ref{ex:NIMwinconditions}. We seek this kind of formula, and note that it can be used to obtain an $O(1)$ algorithm.
\par In the following sections, we will more rigorously define the win function, as well as introduce the terms \textit{lower class}, \textit{middle class}, and \textit{upper class}. We will prove important theorems concerning these terms. In the last three sections, we obtain and prove win conditions for $A=\{1,L\},$ then $A=\{1,L,L+1\}$ ($L$ even), and $A=\{1,L,L+1\}$ ($L$ odd).

There has been one paper on NIM with Cash before~\cite{nwcold}. That paper and
this one have some overlap. The cases where either one player is  rich or both
are poor (which are easy cases) are in both papers. For the case of both players being
either rich or poor, that paper has an abstract approach which is complicated and
hard to use. This paper instead gives concrete criteria for particular examples.

\section{The Win Function}
In order to make our definition of the win function more rigorous, we will explicitly define it in this section. 
\par\noindent 
\begin{definition} Define a function $W:\mathbb{N}^3 \rightarrow \{1,2\}$ that maps the number of stones in the pile and the amount of money each player has to the player who wins ($1$ or $2$).
\end{definition}
\begin{notation} If Player 1 wins, then we write $W(n;d,e)=1$ for the NIM with Cash game (with $A$ understood), and $W(n)$ for the regular NIM game. Similar for Player 2 wins. Should we wish to conveniently describe the losing player, we will use $3-W(n;d,e)$ or $3-W(n).$
\end{notation}
\section{What is Upper Class?}
We are mainly interested with the following question: What is the minimum amount of money a player needs to win normally? 
\par
Intuitively, we want to say that a player is \textit{upper class} if they have enough money to be able to take stones from the pile the whole game and play using normal strategy. However, this is problematic because this intuition only makes sense for the winning player. We will define upper class for the losing player as well; this will be helpful for two reasons:
\begin{enumerate}
    \item If the losing player (in normal NIM) is upper class, and the winning player (in normal NIM) is not, it turns out that the losing player wins NIM with Cash.
    \item It will be useful for induction proofs later.
\end{enumerate}
\begin{definition}Let $A \subseteq^{\mathrm{fin}} \mathbb{N}.$ Define a function $U:\mathbb{N} \rightarrow \mathbb{N}$ that maps the number of stones in the pile to the minimum amount of money a player needs to be upper class. Let $U_1(n)$ and $U_2(n)$ denote these amounts for Players 1 and 2, respectively.
\end{definition}
\par
We now present a definition of the upper class function for both players, which includes the fix we mentioned earlier.
\begin{definition}\label{def:upper} Let $A \subseteq^{\mathrm{fin}} \mathbb{N}$ and let $n\in\mathbb{N}.$ Then,
\begin{itemize}
    \item If $n<\min(A),$ then $U_1(n)=U_2(n)=0$ (Player 1 cannot move, so Player 2 wins by default).
    \item If $W(n)=1,$ then 
    \begin{itemize}
        \item $U_1(n)$ is defined to be the least value of $U_2(n-a)+a$ across all $a\in A,a\leq n$ such that $W(n-a)=2.$
        \item $U_2(n)$ is defined to be the greatest value of $U_1(n-a)$ across all $a\in A,a\leq n.$
    \end{itemize}
    \item If $W(n)=2,$ then
    \begin{itemize}
        \item $U_1(n)$ is defined to be the least value of $U_2(n-a)+a$ across all $a\in A,a\leq n$ such that $U_1(n-a)=U_2(n).$
        \item $U_2(n)$ is defined to be the greatest value of $U_1(n-a)$ across all $a\in A,a\leq n.$
    \end{itemize}
\end{itemize}
\end{definition}
It is easy to check by induction on $n$ that this definition is equivalent to our intuition for winning players. Note also that our definition for $U_1(n)$ when $W(n)=2$ can be reformulated as the maximum number of stones that Player 1 can take from the pile over the course of the game across all possible cases of both players playing normally. For example, when $A=\{1,3,4\},U_1(9)=6$ (Player 2 wins) because Player 1 can take $1$ on his first move, $4$ on his second move, and $1$ on his last move; it can be easily verified that this is the optimal case, and that both players are playing normally. Similar for $U_2(n)$ when $W(n)=1.$
\par
Now, we will present two theorems (which are straightfoward proofs by induction on $n$) to answer our question for $A=\{1,L\},L$ even and $\{1,L,L+1\}.$ Note that the $A=\{1,L\}$ case for $L$ odd essentially reduces to $A={1}.$ Let $U_{\text{win}}(n)$ denote $U_1(n)$ when $W(n)=1$ or $U_2(n)$ when $W(n)=2.$ Let $U_{\text{lose}}(n)$ denote $U_1(n)$ when $W(n)=2$ or $U_2(n)$ when $W(n)=1.$ 
\\
\begin{theorem} For $A=\{1,L\},L$ even:
\begin{itemize}
    \item $U_{\text{win}}(n)=L\bigg\lfloor\dfrac{n}{L+1}\bigg\rfloor,$ for $n\equiv0\pmod{L+1}.$
    \item $U_{\text{win}}(n)=L\bigg\lfloor\dfrac{n}{L+1}\bigg\rfloor+1,$ for $n\equiv1\pmod{L+1}.$
    \item $U_{\text{win}}(n)=L\bigg\lfloor\dfrac{n}{L+1}\bigg\rfloor+1,$ for $n\equiv2\pmod{L+1}.$
    \item $U_{\text{win}}(n)=L\bigg\lfloor\dfrac{n}{L+1}\bigg\rfloor+2,$ for $n\equiv3\pmod{L+1}.$
    \item $U_{\text{win}}(n)=L\bigg\lfloor\dfrac{n}{L+1}\bigg\rfloor+2,$ for $n\equiv4\pmod{L+1}.$
    \\ \\
    $\vdots$
    \\
    \item $U_{\text{win}}(n)=L\bigg\lfloor\dfrac{n}{L+1}\bigg\rfloor+\bigg\lfloor \dfrac{L-2}{2} \bigg\rfloor,$ for $n\equiv L-2 \pmod{L+1}.$ 
    \item $U_{\text{win}}(n)=L\bigg\lfloor\dfrac{n}{L+1}\bigg\rfloor+\bigg\lfloor \dfrac{L-1}{2} \bigg\rfloor,$ for $n\equiv L-1 \pmod{L+1}.$
    \item $U_{\text{win}}(n)=L\bigg(\bigg\lfloor \dfrac{n}{L+1} \bigg\rfloor+1\bigg),$  for $n\equiv L\pmod{L+1},$
\end{itemize}
or equivalently, writing $n=k(L+1)+i,$ where $0\leq i<L+1$ gives us:
\begin{itemize}
    \item If $i<L,$ then $U_{\text{win}}(n)=Lk+\bigg\lceil \dfrac{i}{2} \bigg\rceil.$
    \item If $i=L,$ then $U_{\text{win}}(n)=L(k+1).$
\end{itemize}
Similarly, 
\begin{itemize}
    \item If $n<L,$ then $U_{\text{lose}}(n)=\bigg\lfloor \dfrac{n}{2} \bigg\rfloor.$
    \item If $n\geq L, i<L,$ then $U_{\text{lose}}(n)=L\bigg(k-\dfrac{1}{2}\bigg)+\bigg\lfloor \dfrac{i}{2} \bigg\rfloor+1.$
    \item If $i=L,$ then $U_{\text{lose}}(n)=L\bigg(k+\dfrac{1}{2}\bigg).$
\end{itemize}
\end{theorem}
\begin{theorem} For $A=\{1,L,L+1\},L$ even:
    \begin{itemize}
        \item $U_{\text{win}}(n)=\dfrac{3}{2}L\bigg\lfloor \dfrac{n}{2L} \bigg\rfloor,$ for $n\equiv 0\pmod{2L}.$
        \item $U_{\text{win}}(n)=\dfrac{3}{2}L\bigg\lfloor \dfrac{n}{2L} \bigg\rfloor+1,$ for $n\equiv 1\pmod{2L}.$
        \item $U_{\text{win}}(n)=\dfrac{3}{2}L\bigg\lfloor \dfrac{n}{2L} \bigg\rfloor+1,$ for $n\equiv 2\pmod{2L}.$
        \item $U_{\text{win}}(n)=\dfrac{3}{2}L\bigg\lfloor \dfrac{n}{2L} \bigg\rfloor+2,$ for $n\equiv 3\pmod{2L}.$
        \item $U_{\text{win}}(n)=\dfrac{3}{2}L\bigg\lfloor \dfrac{n}{2L} \bigg\rfloor+2,$ for $n\equiv 4\pmod{2L}.$ \\ \\
        $\vdots$
        \\
        \item $U_{\text{win}}(n)=\dfrac{3}{2}L\bigg\lfloor \dfrac{n}{2L} \bigg\rfloor+\bigg\lceil \dfrac{L-2}{2} \bigg\rceil,$ for $n\equiv L-2\pmod{2L}.$
        \item $U_{\text{win}}(n)=\dfrac{3}{2}L\bigg\lfloor \dfrac{n}{2L} \bigg\rfloor+\bigg\lceil \dfrac{L-1}{2} \bigg\rceil,$ for $n\equiv L-1\pmod{2L}.$
        \item $U_{\text{win}}(n)=\dfrac{3}{2}L\bigg\lfloor \dfrac{n}{2L} \bigg\rfloor+L,$ for $n\equiv L\pmod{2L}.$
        \item $U_{\text{win}}(n)=\dfrac{3}{2}L\bigg\lfloor \dfrac{n}{2L} \bigg\rfloor+L+1,$ for $n\equiv L+1\pmod{2L}$
        \item $U_{\text{win}}(n)=\dfrac{3}{2}L\bigg\lfloor \dfrac{n}{2L} \bigg\rfloor+L+1,$ for $n\equiv L+2\pmod{2L}$
        \item $U_{\text{win}}(n)=\dfrac{3}{2}L\bigg\lfloor \dfrac{n}{2L} \bigg\rfloor+L+2,$ for $n\equiv L+3\pmod{2L}$
        \item $U_{\text{win}}(n)=\dfrac{3}{2}L\bigg\lfloor \dfrac{n}{2L} \bigg\rfloor+L+2,$ for $n\equiv L+4\pmod{2L}$ \\ \\
        $\vdots$
        \\
        \item $U_{\text{win}}(n)=\dfrac{3}{2}L\bigg\lfloor \dfrac{n}{2L} \bigg\rfloor+L+\bigg\lceil \dfrac{L-2}{2} \bigg\rceil,$ for $n\equiv 2L-2\pmod{2L}.$
        \item $U_{\text{win}}(n)=\dfrac{3}{2}L\bigg\lfloor \dfrac{n}{2L} \bigg\rfloor+L+\bigg\lceil \dfrac{L-1}{2} \bigg\rceil,$ for $n\equiv 2L-1\pmod{2L},$
    \end{itemize}
    or equivalently, writing $n=2Lk+i,$ where $0\leq i<2L$ gives us:
    \begin{itemize}
        \item If $0\leq i<L,$ then $U_{\text{win}}(n)=\dfrac{3Lk}{2}+\bigg\lceil \dfrac{i}{2} \bigg\rceil.$
        \item If $L\leq i\leq 2L-1,$ then $U_{\text{win}}(n)=\dfrac{3Lk}{2}+L+\bigg\lceil \dfrac{i-L}{2} \bigg\rceil.$
    \end{itemize}
    Similarly, 
    \begin{itemize}
        \item If $0\leq i<L+1,$ then $U_{\text{lose}}(n)=\dfrac{3Lk}{2}+\bigg\lfloor \dfrac{i}{2} \bigg\rfloor.$
        \item If $L+1\leq i\leq 2L-1,$ then $U_{\text{lose}}(n)=\dfrac{3Lk}{2}+L+\bigg\lfloor \dfrac{i-L}{2} \bigg\rfloor.$
    \end{itemize}
For $A=\{1,L,L+1\},L$ odd:
    \begin{itemize}
        \item $U_{\text{win}}(n)=\dfrac{3L+1}{2}\bigg\lfloor \dfrac{n}{2L+1} \bigg\rfloor,$ for $n\equiv 0\pmod{2L+1}.$
        \item $U_{\text{win}}(n)=\dfrac{3L+1}{2}\bigg\lfloor \dfrac{n}{2L+1} \bigg\rfloor+1,$ for $n\equiv 1\pmod{2L+1}.$
        \item $U_{\text{win}}(n)=\dfrac{3L+1}{2}\bigg\lfloor \dfrac{n}{2L+1} \bigg\rfloor+1,$ for $n\equiv 2\pmod{2L+1}.$
        \item $U_{\text{win}}(n)=\dfrac{3L+1}{2}\bigg\lfloor \dfrac{n}{2L+1} \bigg\rfloor+2,$ for $n\equiv 3\pmod{2L+1}.$
        \item $U_{\text{win}}(n)=\dfrac{3L+1}{2}\bigg\lfloor \dfrac{n}{2L+1} \bigg\rfloor+2,$ for $n\equiv 4\pmod{2L+1}.$ \\ \\
        $\vdots$
        \\
        \item $U_{\text{win}}(n)=\dfrac{3L+1}{2}\bigg\lfloor \dfrac{n}{2L+1} \bigg\rfloor+\bigg\lceil \dfrac{L-2}{2} \bigg\rceil,$ for $n\equiv L-2\pmod{2L+1}.$
        \item $U_{\text{win}}(n)=\dfrac{3L+1}{2}\bigg\lfloor \dfrac{n}{2L+1} \bigg\rfloor+\bigg\lceil \dfrac{L-1}{2} \bigg\rceil,$ for $n\equiv L-1\pmod{2L+1}.$
        \item $U_{\text{win}}(n)=\dfrac{3L+1}{2}\bigg\lfloor \dfrac{n}{2L+1} \bigg\rfloor+\bigg\lceil \dfrac{L}{2} \bigg\rceil,$ for $n\equiv L\pmod{2L+1}.$
        \item $U_{\text{win}}(n)=\dfrac{3L+1}{2}\bigg\lfloor \dfrac{n}{2L+1} \bigg\rfloor+L+1,$ for $n\equiv L+1\pmod{2L+1}$
        \item $U_{\text{win}}(n)=\dfrac{3L+1}{2}\bigg\lfloor \dfrac{n}{2L+1} \bigg\rfloor+L+1,$ for $n\equiv L+2\pmod{2L+1}$
        \item $U_{\text{win}}(n)=\dfrac{3L+1}{2}\bigg\lfloor \dfrac{n}{2L+1} \bigg\rfloor+L+2,$ for $n\equiv L+3\pmod{2L+1}$
        \item $U_{\text{win}}(n)=\dfrac{3L+1}{2}\bigg\lfloor \dfrac{n}{2L+1} \bigg\rfloor+L+2,$ for $n\equiv L+4\pmod{2L+1}$ \\ \\
        $\vdots$
        \\
        \item $U_{\text{win}}(n)=\dfrac{3L+1}{2}\bigg\lfloor \dfrac{n}{2L+1} \bigg\rfloor+L+\bigg\lceil \dfrac{L-1}{2} \bigg\rceil,$ for $n\equiv 2L-1\pmod{2L+1}.$
        \item $U_{\text{win}}(n)=\dfrac{3L+1}{2}\bigg\lfloor \dfrac{n}{2L+1} \bigg\rfloor+L+\bigg\lceil \dfrac{L}{2} \bigg\rceil,$ for $n\equiv 2L\pmod{2L+1},$
    \end{itemize}
    or equivalently, writing $n=(2L+1)k+i,$ where $0\leq i<2L+1$ gives us:
    \begin{itemize}
        \item If $0\leq i<L+1,$ then $U_{\text{win}}(n)=\dfrac{(3L+1)k}{2}+\bigg\lceil \dfrac{i}{2} \bigg\rceil.$
        \item If $L+1\leq i\leq 2L,$ then $U_{\text{win}}(n)=\dfrac{(3L+1)k}{2}+L+\bigg\lceil \dfrac{i-L}{2} \bigg\rceil.$
    \end{itemize}
    Similarly, 
    \begin{itemize}
        \item If $0\leq i<L+2,$ then $U_{\text{lose}}(n)=\dfrac{(3L+1)k}{2}+\bigg\lfloor \dfrac{i}{2} \bigg\rfloor.$
        \item If $L+2\leq i\leq 2L,$ then $U_{\text{lose}}(n)=\dfrac{(3L+1)k}{2}+L+\bigg\lfloor \dfrac{i-L}{2} \bigg\rfloor.$
    \end{itemize}
\end{theorem}
The following lemma is very useful and easy to verify.
\begin{lemma}\label{lem:upperclasswinner}If $W(n)=1,$ then $W(n;U(n),\infty)=1.$ If $W(n)=2,$ then $W(n;\infty,U(n))=2.$ \end{lemma}
\section{What is Middle Class?}
\noindent How much money does a player have if he has less than $U$ dollars, but enough money to last the whole game? We will rigorously define this quantity for this section.
%BILL TO DOUG Need Middle Class to be defined for a player for both the case where he is a winner and he is a loswer.
%BILL TO DOUB Need Middle Class to be defined for a player IND of the other player.
\begin{definition}\label{def:middle} We say that Player $1$ is \textit{middle class} if both of the following hold:
\begin{itemize}
    \item Player 1 is not upper class.
    \item Player 1 will be able to finish the game no matter what Player 2 does if he takes $\min(A)$ the entire game.
\end{itemize}
Similar for Player $2.$
\end{definition}
\begin{definition} Let $A \subseteq^{\mathrm{fin}} \mathbb{N}.$ Define a function $M:\mathbb{N} \rightarrow \mathbb{N}$ that maps the number of stones in the pile to the minimum amount of money a player needs to be middle class. Let $M_1(n)$ and $M_2(n)$ denote these amounts for Players 1 and 2, respectively.
\end{definition}
The following theorem follows from our definition of middle class. and can be applied to our sets $A=\{1,L\}$ and $A=\{1,L,L+1\}$ ($L$ can be even or odd). 
\begin{theorem}
Let $A=\{1,L\}$ or $\{1,L,L+1\}.$ Then,
\begin{itemize}
    \item If $n\equiv 0\pmod2:$
    \begin{itemize}
        \item $M_1(n)=\dfrac{n}{2}+1$ and $M_2(n)=\dfrac{n}{2}.$
    \end{itemize}
    \item If $n\equiv 1\pmod2:$
    \begin{itemize}
        \item $M_1(n)=M_2(n)=\dfrac{n+1}{2}.$
    \end{itemize}
\end{itemize}
\end{theorem}
\section{What is Lower Class?}
It is possible for players to have only enough money to play miserly. First, some more definitions.
\par\noindent
\begin{definition}\label{def:lower} We say a player is \textit{lower class} if
\begin{itemize}
\item he is not middle class.
\item he goes broke if he does not take $\min(A)$ each turn.
\end{itemize}
\end{definition}
We will simply set the lower class money amount for Player $1$ as every nonnegative integer less than $M_1(n)$ and similarly, every nonnegative integer less than $M_2(n)$ for Player $2.$ The following theorem, which is intuitively obvious, makes this more rigorous.
\begin{theorem}
Let the variable $P_1\in\mathbb{N},$ describe the set of all possible values for which Player 1 is poor, and similarly for Player 2. Then, we have the inequalities
\begin{center}
    $0\leq P_1<M_1(n)$
\end{center}
and
\begin{center}
    $0\leq P_2<M_2(n).$
\end{center}
\end{theorem}
\section{Win Conditions}
Here, we will employ our definitions to present win conditions for NIM with Cash for the sets $A=\{1,L\}$ and $A=\{1,L,L+1\}.$
\par\noindent
The following two theorems can be applied to both of our sets. 
\par\noindent
\begin{theorem}
Let $A \subseteq^{\mathrm{fin}} \mathbb{N}.$ If at least one player is upper class, then

\begin{enumerate}
    \item If $d \geq U_1(n)$ and $e<U_2(n),$ then $W(n;d,e)=1.$
    \item If $d<U_1(n)$ and $e \geq U_2(n),$ then $W(n;d,e)=2.$
    \item If $d \geq U_1(n)$ and $e \geq U_2(n),$ then $W(n;d,e)=W(n).$
\end{enumerate}
\end{theorem}
\begin{proof}
We will prove all parts using an induction on $n,$ Definition \ref{def:upper}, Lemma \ref{lem:upperclasswinner}, and a key intuition.
\begin{enumerate}
    \item We first give the intuition. Observe that if $W(n)=2,$ Player 2 has less than $U_1(n)$ dollars, and Player 1 has at least $U_2(n)$ dollars, Player $2$ will go broke if he never switches his strategy from normal. Thus, Player $1$ can take $\min{(A)}$ after each time Player $2$ plays a normal move; when Player $2$ decides to switch strategies, Player $1$ is then in a winning position and has enough money to win. Formally, the proof uses Definition \ref{def:upper} and an induction on $n.$ 
    \item The proof has the same intuition and formality as $(1).$
    \item Intuitively this holds, as upper class for the winning player means that he will be able to play normally and win the game.
\end{enumerate}\end{proof}
\begin{theorem} Let $A \subseteq^{\mathrm{fin}} \mathbb{N}.$ If at least one player is lower class, then
\begin{enumerate}
    \item If $d \geq M_1(n)$ and $e<M_2(n),$ then $W(n;d,e)=1.$ \label{thm6.2:part1} 
    \item If $d<M_1(n)$ and $e \geq M_2(n),$ then $W(n;d,e)=2.$
    \item If $d=e<M_1(n),$ then $W(n;d,e)=2.$ 
    \item If $e<d<M_1(n),$ then $W(n;d,e)=1.$ \label{thm6.2:part4}
    \item If $d<e<M_2(n),$ then $W(n;d,e)=2.$
\end{enumerate}
\end{theorem}
\begin{proof} Applying definitions \ref{def:middle} and \ref{def:lower}, we see that Player $2$ will go broke first in (\ref{thm6.2:part1}) and (\ref{thm6.2:part4}), and that Player $1$ will go broke first in the remaining two cases. The result follows. \end{proof}
\par
The following three theorems describe win conditions for when both players are middle class, and can also be referred to as the staircase theorems. Once we are finished presenting the theorems and proofs, we will provide diagrams to explain this name. Recall for these theorems that $M_1(n)=\dfrac{n}{2}+1, M_2(n)=\dfrac{n}{2}$ for even $n$ and $M_1(n)=M_2(n)=\dfrac{n+1}{2}$ for odd $n.$
\begin{theorem}[Staircase Theorem for $\{1,2\}$]
If $A=\{1,2\},$ we can determine who wins using the following statements.
\begin{enumerate}
    \item $n\equiv 0\pmod{6}:$
    \begin{itemize}
        \item If $d<M_1(n)+1$ and $e<M_2(n)+1,$ then $W(n;d,e)=1.$
        \item If $d<M_1(n)+1$ and $e\geq M_2(n)+1,$ then $W(n;d,e)=2.$
        \item If $M_1(n)+1\leq d<M_1(n)+2$ and $e<M_2(n)+2,$ then $W(n;d,e)=1.$
        \item If $M_1(n)+1\leq d<M_1(n)+2$ and $e\geq M_2(n)+2,$ then $W(n;d,e)=2.$
        \item If $M_1(n)+2\leq d<M_1(n)+3$ and $e<M_2(n)+3,$ then $W(n;d,e)=1.$
        \item If $M_1(n)+2\leq d<M_1(n)+3$ and $e\geq M_2(n)+3,$ then $W(n;d,e)=2.$
        \\
        \\
        $\vdots$
        \\
        \item If $M_1(n)+k\leq d<M_1(n)+k+1$ and $e<M_2(n)+k+1,$ then $W(n;d,e)=1.$
        \item If $M_1(n)+k\leq d<M_1(n)+k+1$ and $e\geq M_2(n)+k+1,$ then $W(n;d,e)=2.$
        \\
        \\
        $\vdots$
        \\
        \item If $d\geq U_1(n)-1$ and $e<U_2(n)-1,$ then $W(n;d,e)=1.$
        \item If $d\geq U_1(n)-1$ and $e\geq U_2(n)-1,$ then $W(n;d,e)=2.$
    \end{itemize}
    \item $n\equiv 1,5\pmod{6}:$
    \begin{itemize}
        \item If $d<M_1(n)+1,$ then $W(n;d,e)=2.$
        \item If $M_1(n)+1\leq d<M_1(n)+2$ and $e<M_2(n)+1,$ then $W(n;d,e)=1.$
        \item If $M_1(n)+1\leq d<M_1(n)+2$ and $e\geq M_2(n)+1,$ then $W(n;d,e)=2.$
        \item If $M_1(n)+2\leq d<M_1(n)+3$ and $e<M_2(n)+2,$ then $W(n;d,e)=1.$
        \item If $M_1(n)+2\leq d<M_1(n)+3$ and $e\geq M_2(n)+2,$ then $W(n;d,e)=2.$
        \item If $M_1(n)+3\leq d<M_1(n)+4$ and $e<M_2(n)+3,$ then $W(n;d,e)=1.$
        \item If $M_1(n)+3\leq d<M_1(n)+4$ and $e\geq M_2(n)+3,$ then $W(n;d,e)=2.$
        \\
        \\
        $\vdots$
        \\
        \item If $M_1(n)+k+1\leq d<M_1(n)+k+2$ and $e<M_2(n)+k+1,$ then $W(n;d,e)=1.$
        \item If $M_1(n)+k+1\leq d<M_1(n)+k+2$ and $e\geq M_2(n)+k+1,$ then $W(n;d,e)=2.$
        \\
        \\
        $\vdots$
        \\
        \item If $d\geq U_1(n)-1,$ then $W(n;d,e)=1.$
    \end{itemize}
    \item $n\equiv 2,4\pmod{6}:$
    \begin{itemize}
        \item If $d<M_1(n)+1$ and $e<M_2(n)+1,$ then $W(n;d,e)=1.$
        \item If $d<M_1(n)+1$ and $e\geq M_2(n)+1,$ then $W(n;d,e)=2.$
        \item If $M_1(n)+1\leq d<M_1(n)+2$ and $e<M_2(n)+2,$ then $W(n;d,e)=1.$
        \item If $M_1(n)+1\leq d<M_1(n)+2$ and $e\geq M_2(n)+2,$ then $W(n;d,e)=2.$
        \item If $M_1(n)+2\leq d<M_1(n)+3$ and $e<M_2(n)+3,$ then $W(n;d,e)=1.$
        \item If $M_1(n)+2\leq d<M_1(n)+3$ and $e\geq M_2(n)+3,$ then $W(n;d,e)=2.$
        \\
        \\
        $\vdots$
        \\
        \item If $M_1(n)+k\leq d<M_1(n)+k+1$ and $e<M_2(n)+k+1,$ then $W(n;d,e)=1.$
        \item If $M_1(n)+k\leq d<M_1(n)+k+1$ and $e\geq M_2(n)+k+1,$ then $W(n;d,e)=2.$
        \\
        \\
        $\vdots$
        \\
        \item If $d\geq U_1(n)-1,$ then $W(n;d,e)=1.$
    \end{itemize}
    \item $n\equiv 3\pmod{6}:$
    \begin{itemize}
        \item If $d<M_1(n)+1,$ then $W(n;d,e)=2.$
        \item If $M_1(n)+1\leq d<M_1(n)+2$ and $e<M_2(n)+1,$ then $W(n;d,e)=1.$
        \item If $M_1(n)+1\leq d<M_1(n)+2$ and $e\geq M_2(n)+1,$ then $W(n;d,e)=2.$
        \item If $M_1(n)+2\leq d<M_1(n)+3$ and $e<M_2(n)+2,$ then $W(n;d,e)=1.$
        \item If $M_1(n)+2\leq d<M_1(n)+3$ and $e\geq M_2(n)+2,$ then $W(n;d,e)=2.$
        \item If $M_1(n)+3\leq d<M_1(n)+4$ and $e<M_2(n)+3,$ then $W(n;d,e)=1.$
        \item If $M_1(n)+3\leq d<M_1(n)+4$ and $e\geq M_2(n)+3,$ then $W(n;d,e)=2.$
        \\
        \\
        $\vdots$
        \\
        \item If $M_1(n)+k+1\leq d<M_1(n)+k+2$ and $e<M_2(n)+k+1,$ then $W(n;d,e)=1.$
        \item If $M_1(n)+k+1\leq d<M_1(n)+k+2$ and $e\geq M_2(n)+k+1,$ then $W(n;d,e)=2.$
        \\
        \\
        $\vdots$
        \\
        \item If $d\geq U_1(n)-1$ and $e<U_2(n)-1,$ then $W(n;d,e)=1.$
        \item If $d\geq U_1(n)-1$ and $e\geq U_2(n)-1,$ then $W(n;d,e)=2.$
    \end{itemize}
\end{enumerate}
\end{theorem}
\begin{proof}We will prove all of the statements using induction on $n$ and $n-1.$ It is easy to check the base cases. Observe that cases in which $n$ is even have essentially the same win conditions, apart from the ending subcases, which are trivial to prove. Similar for when $n$ is odd.
\par\noindent
Without loss of generality, let $n\equiv0 \pmod2.$ Assume that for every even integer up to $n:$ \begin{itemize}
    \item if $M_1(n)+k\leq d<M_1(n)+k+1,$ and $e<M_2(n)+k+1,$ then $W(n;d,e)=1.$ \hfill (1)
    \item if $M_1(n)+k\leq d<M_1(n)+k+1,$ and $e\geq M_2(n)+k+1,$ then $W(n;d,e)=2.$ \hfill (2)
\end{itemize}  
Moreover, assume that for every odd integer up to $n-1:$ 
\begin{itemize}
    \item if $M_1(n)+k+1\leq d<M_1(n)+k+2,$ and $e<M_2(n)+k+1,$ then $W(n;d,e)=1.$ \hfill(3)
    \item if $M_1(n)+k+1\leq d<M_1(n)+k+2,$ and $e\geq M_2(n)+k+1,$ then $W(n;d,e)=2.$ \hfill(4)
\end{itemize}
\par\noindent
Then we need to prove that the equivalent claims for $n+2$ and $n+1$ stones are true, respectively. 
\par\noindent
For $n+2$ stones, let us examine what happens when
\begin{align*}
    \frac{n}{2}+k+2\leq d<\frac{n}{2}+k+3 \ \text{and} \ e<\frac{n}{2}+k+2,
\end{align*}
where we have used our definitions for $M_1(n)$ and $M_2(n).$
\par\noindent
If both players each take $1$ on their first moves, we are back to the game where there are $n$ stones, $M_1(n)+k\leq d<M_1(n)+k+1,$ and $e<M_2(n)+k+1,$ for which we know that $W(n;d,e)=1.$
\par\noindent
If Player $1$ takes $1,$ but Player $2$ takes $2,$ both on their first moves, then 
\begin{align*}
    \frac{n}{2}+k+1\leq d<\frac{n}{2}+k+2 \ \text{and} \ e<\frac{n}{2}+k.
\end{align*}
Note that there are $n-1$ stones, so applying inductive hypothesis (3) yields
\begin{align*}
    \text{if \ } \frac{n}{2}+k+1 \leq d<\frac{n}{2}+k+2 \text{ and } e<\frac{n}{2}+k+1, \text{ then } W(n;d,e)=1.
\end{align*}
Comparing the two conditions yields the conclusion that Player $1$ wins the game even if Player $2$ takes $2,$ which completes part of the induction.
\par\noindent
For $n+2$ stones, now let us examine what happens when
\begin{align*}
    \frac{n}{2}+k+2\leq d<\frac{n}{2}+k+3 \ \text{and} \ e\geq\frac{n}{2}+k+2.
\end{align*}
If both players take $1$ on their first move, then it follows by the same reasoning as before that $W(n;d,e)=2.$
\par\noindent
By using the same reasoning, if Player $1$ now takes $2,$ Player $2$ can choose to take $2,$ so Player $2$ also wins, completing another part of the induction.
\par\noindent
For $n+1$ stones, now let us examine what happens when
\begin{align*}
    \frac{n}{2}+k+2\leq d<\frac{n}{2}+k+3 \ \text{and} \ e<\frac{n}{2}+k+2.
\end{align*}
If both players take $1$ on their first move, it follows that Player $1$ wins by the same reasoning as before.
\par\noindent
If Player $1$ takes 1, but Player $2$ takes $2,$ both on their first moves, then 
\begin{align*}
    \frac{n}{2}+k+1\leq d<\frac{n}{2}+k+2 \ \text{and} \ e<\frac{n}{2}+k.
\end{align*}
Note that there are $n-2$ stones, so applying inductive hypothesis (1) yields
\begin{align*}
    \text{if \ } \frac{n}{2}+k \leq d<\frac{n}{2}+k+1 \text{ and } e<\frac{n}{2}+k, \text{  then } W(n;d,e)=1.
\end{align*}
Note that Player $1$ having more money does not affect the outcome of the game in this case. Comparing the two conditions yields the conclusion that Player $1$ wins the game even if Player $2$ takes $2,$ which completes another part of the induction.
\par\noindent
For $n+1$ stones, now let us examine what happens when
\begin{align*}
    \frac{n}{2}+k+2\leq d<\frac{n}{2}+k+3 \ \text{and} \ e\geq\frac{n}{2}+k+2.
\end{align*}
If both players take $1$ on their first move, then it follows that $W(n;d,e)=2.$
\par\noindent
If Player $1$ now takes $2,$ Player $2$ can choose to take $2,$ and it follows that Player $2$ also wins, completing the induction. \end{proof}
\begin{theorem}[Staircase Theorem for $\{1,L\},L$ Even and $\geq 4$]
If $A=\{1,L\},L$ even and $\geq 4,$ we can determine who wins using the following statements.
\begin{enumerate}
    \item $n \equiv 0,2,4,\cdots,2k,\cdots,L-2\pmod{2L+2}:$
    \begin{itemize}
        \item If $d<M_1(n)+\frac{L}{2}-1,$ then $W(n;d,e)=2.$
        \item If $M_1(n)+\frac{L}{2}-1\leq d<M_1(n)+\frac{3L}{2}-2$ and $e<M_2(n)+L-1,$ then $W(n;d,e)=1.$
        \item If $M_1(n)+\frac{L}{2}-1\leq d<M_1(n)+\frac{3L}{2}-2$ and $e\geq M_2(n)+L-1,$ then $W(n;d,e)=2.$
        \item If $M_1(n)+\frac{3L}{2}-2\leq d<M_1(n)+\frac{5L}{2}-3$ and $e<M_2(n)+2L-2,$ then $W(n;d,e)=1.$
        \item If $M_1(n)+\frac{3L}{2}-2\leq d<M_1(n)+\frac{5L}{2}-3$ and $e\geq M_2(n)+2L-2,$ then $W(n;d,e)=2.$
        \\
        \\
        $\vdots$
        \\
        \item If $M_1(n)+(L-1)k+\frac{L}{2}-1\leq d<M_1(n)+(L-1)k+\frac{3L}{2}-2$ and $e<M_2(n)+(L-1)k+L-1,$ then $W(n;d,e)=1.$
        \item If $M_1(n)+(L-1)k+\frac{L}{2}-1\leq d<M_1(n)+(L-1)k+\frac{3L}{2}-2$ and $e\geq M_2(n)+(L-1)k+L-1,$ then $W(n;d,e)=2.$
        \\
        \\
        $\vdots$
        \\
        \item If $d\geq U_1(n)-L+1$ and $e<U_2(n)-L+1,$ then $W(n;d,e)=1.$
        \item If $d\geq U_1(n)-L+1$ and $e\geq U_2(n)-L+1,$ then $W(n;d,e)=2.$
    \end{itemize}
    \item $n\equiv 1,3,5,\cdots,2k+1,\cdots,L-1,2L+1\pmod{2L+2}:$
    \begin{itemize}
        \item If $d<M_1(n)+L-1$ and $e<M_2(n)+\frac{L}{2}-1,$ then $W(n;d,e)=2.$
        \item If $d<M_1(n)+L-1$ and $e\geq M_2(n)+\frac{L}{2}-1,$ then $W(n;d,e)=2.$
        \item If $M_1(n)+L-1\leq d<M_1(n)+2L-2$ and $e<M_2(n)+\frac{3L}{2}-2,$ then $W(n;d,e)=1.$
        \item If $M_1(n)+L-1\leq d<M_1(n)+2L-2$ and $e\geq M_2(n)+\frac{3L}{2}-2,$ then $W(n;d,e)=2.$
        \item If $M_1(n)+2L-2\leq d<M_1(n)+3L-3$ and $e<M_2(n)+\frac{5L}{2}-3,$ then $W(n;d,e)=1.$
        \item If $M_1(n)+2L-2\leq d<M_1(n)+3L-3$ and $e\geq M_2(n)+\frac{5L}{2}-3,$ then $W(n;d,e)=2.$
        \\
        \\
        $\vdots$
        \\
        \item If $M_1(n)+(L-1)k\leq d<M_1(n)+(L-1)k+L-1$ and $e<M_2(n)+(L-1)k+\frac{L}{2}-1,$ then $W(n;d,e)=1.$
        \item If $M_1(n)+(L-1)k\leq d<M_1(n)+(L-1)k+L-1$ and $e\geq M_2(n)+(L-1)k+\frac{L}{2}-1,$ then $W(n;d,e)=2.$
        \\
        \\
        $\vdots$
        \\
        \item If $d\geq U_1(n)-L+1,$ then $W(n;d,e)=1.$
    \end{itemize}
    \item $n\equiv L,L+2,L+4,\cdots,L+2k,\cdots,2L \pmod{2L+2}:$
    \begin{itemize}
        \item If $d<M_1(n)+\frac{L}{2}-1,$ then $W(n;d,e)=2.$
        \item If $M_1(n)+\frac{L}{2}-1\leq d<M_1(n)+\frac{3L}{2}-2$ and $e<M_2(n)+L-1,$ then $W(n;d,e)=1.$
        \item If $M_1(n)+\frac{L}{2}-1\leq d<M_1(n)+\frac{3L}{2}-2$ and $e\geq M_2(n)+L-1,$ then $W(n;d,e)=2.$
        \item If $M_1(n)+\frac{3L}{2}-2\leq d<M_1(n)+\frac{5L}{2}-3$ and $e<M_2(n)+2L-2,$ then $W(n;d,e)=1.$
        \item If $M_1(n)+\frac{3L}{2}-2\leq d<M_1(n)+\frac{5L}{2}-3$ and $e\geq M_2(n)+2L-2,$ then $W(n;d,e)=2.$
        \\
        \\
        $\vdots$
        \\
        \item If $M_1(n)+(L-1)k+\frac{L}{2}-1\leq d<M_1(n)+(L-1)k+\frac{3L}{2}-2$ and $e<M_2(n)+(L-1)k+L-1,$ then $W(n;d,e)=1.$
        \item If $M_1(n)+(L-1)k+\frac{L}{2}-1\leq d<M_1(n)+(L-1)k+\frac{3L}{2}-2$ and $e\geq M_2(n)+(L-1)k+L-1,$ then $W(n;d,e)=2.$
        \\
        \\
        $\vdots$
        \\
        \item If $d\geq U_1(n)-L+1,$ then $W(n;d,e)=1.$
    \end{itemize}
    \item $n\equiv L+1,L+3,L+5,\cdots,L+2k+1,\cdots,2L-1 \pmod{2L+2}:$
    \begin{itemize}
        \item If $d<M_1(n)+L-1$ and $e<M_2(n)+\frac{L}{2}-1,$ then $W(n;d,e)=2.$
        \item If $d<M_1(n)+L-1$ and $e\geq M_2(n)+\frac{L}{2}-1,$ then $W(n;d,e)=2.$
        \item If $M_1(n)+L-1\leq d<M_1(n)+2L-2$ and $e<M_2(n)+\frac{3L}{2}-2,$ then $W(n;d,e)=1.$
        \item If $M_1(n)+L-1\leq d<M_1(n)+2L-2$ and $e\geq M_2(n)+\frac{3L}{2}-2,$ then $W(n;d,e)=2.$
        \item If $M_1(n)+2L-2\leq d<M_1(n)+3L-3$ and $e<M_2(n)+\frac{5L}{2}-3,$ then $W(n;d,e)=1.$
        \item If $M_1(n)+2L-2\leq d<M_1(n)+3L-3$ and $e\geq M_2(n)+\frac{5L}{2}-3,$ then $W(n;d,e)=2.$
        \\
        \\
        $\vdots$
        \\
        \item If $M_1(n)+(L-1)k\leq d<M_1(n)+(L-1)k+L-1$ and $e<M_2(n)+(L-1)k+\frac{L}{2}-1,$ then $W(n;d,e)=1.$
        \item If $M_1(n)+(L-1)k\leq d<M_1(n)+(L-1)k+L-1$ and $e\geq M_2(n)+(L-1)k+\frac{L}{2}-1,$ then $W(n;d,e)=2.$
        \\
        \\
        $\vdots$
        \\
        \item If $d\geq U_1(n)-L+1$ and $e<U_2(n)-L+1,$ then $W(n;d,e)=1.$
        \item If $d\geq U_1(n)-L+1$ and $e\geq U_2(n)-L+1,$ then $W(n;d,e)=2.$
    \end{itemize}
\end{enumerate}
\end{theorem}
\begin{proof}We will prove all of the statements using induction on $n$ and $n-1,$ using the same argument in the proof for $\{1,2\}.$ It is easy to check the base cases and the ending subcases.
\par\noindent
Without loss of generality, let $n\equiv0 \pmod2.$ Assume that for every even integer up to $n:$ 
\begin{itemize}
    \item If $M_1(n)+(L-1)k+\frac{L}{2}-1\leq d<M_1(n)+(L-1)k+\frac{3L}{2}-2$ and $e<M_2(n)+(L-1)k+L-1,$ then $W(n;d,e)=1.$ \hfill (5)
    \item If $M_1(n)+(L-1)k+\frac{L}{2}-1\leq d<M_1(n)+(L-1)k+\frac{3L}{2}-2$ and $e\geq M_2(n)+(L-1)k+L-1,$ then $W(n;d,e)=2.$ \hfill (6)
\end{itemize}  Moreover, assume that for every odd integer up to $n-1:$ 
\begin{itemize}
    \item If $M_1(n)+(L-1)k\leq d<M_1(n)+(L-1)k+L-1$ and $e<M_2(n)+(L-1)k+\frac{L}{2}-1,$ then $W(n;d,e)=1.$ \hfill (7)
    \item If $M_1(n)+(L-1)k\leq d<M_1(n)+(L-1)k+L-1$ and $e\geq M_2(n)+(L-1)k+\frac{L}{2}-1,$ then $W(n;d,e)=2.$ \hfill (8)
\end{itemize}
\par\noindent
Then we need to prove that the equivalent claims for $n+2$ and $n+1$ stones are true, respectively. 
\par\noindent
For $n+2$ stones, let us examine what happens when
\begin{align*}
    \frac{n}{2}+2+(L-1)k+\frac{L}{2}-1\leq d<\frac{n}{2}+2+(L-1)k+\frac{3L}{2}-2 \ \text{and} \ e<\frac{n}{2}+1+(L-1)k+\frac{L}{2}-1,
\end{align*}
where we have used our definitions for $M_1(n)$ and $M_2(n).$
\par\noindent
If both players each take $1$ on their first moves, we are back to the game where there are $n$ stones, $M_1(n)+(L-1)k\leq d<M_1(n)+(L-1)k+L-1,$ and $e<M_2(n)+(L-1)k+\frac{L}{2}-1,$ for which we know that $W(n;d,e)=1.$
\par\noindent
If Player $1$ takes $1,$ but Player $2$ takes $L,$ both on their first moves, then 
\begin{align*}
    \frac{n}{2}+(L-1)k+\frac{L}{2}\leq d<\frac{n}{2}+(L-1)k+\frac{3L}{2}-1 \ \text{and} \ e<\frac{n}{2}+(L-1)k-\frac{L}{2}.
\end{align*}
Note that there are $n+1-L$ stones, so applying inductive hypothesis (7) yields
\begin{align*}
    \text{if \ } \frac{n}{2}+(L-1)k-\frac{L}{2}+1\leq d<\frac{n}{2}+(L-1)k+\frac{L}{2} \text{ and } e<\frac{n}{2}+(L-1)k, \text{ then } W(n;d,e)=1.
\end{align*}
Note that in this case, Player $1$ having more money and Player $2$ having less money does not affect the outcome. Comparing the two conditions yields the conclusion that Player $1$ wins the game even if Player $2$ takes $L,$ which completes part of the induction.
\par\noindent
For $n+2$ stones, now let us examine what happens when
\begin{align*}
    \frac{n}{2}+2+(L-1)k+\frac{L}{2}-1\leq d<\frac{n}{2}+2+(L-1)k+\frac{3L}{2}-2 \ \text{and} \ e\geq\frac{n}{2}+1+(L-1)k+\frac{L}{2}-1,
\end{align*}
If both players take $1$ on their first move, then it follows by the same reasoning as before that $W(n;d,e)=2.$
\par\noindent
By using the same reasoning, if Player $1$ now takes $L,$ Player $2$ can choose to take $L,$ so Player $2$ also wins, completing another part of the induction.
\par\noindent
For $n+1$ stones, now let us examine what happens when
\begin{align*}
    \frac{n}{2}+1+(L-1)k\leq d<\frac{n}{2}+1+(L-1)k+L-1 \ \text{and} \ e<\frac{n}{2}+1+(L-1)k+\frac{L}{2}-1.
\end{align*}
If both players take $1$ on their first move, it follows that Player $1$ wins by the same reasoning as before. 
\par\noindent
If Player $1$ takes 1, but Player $2$ takes $L,$ both on their first moves, then 
\begin{align*}
    \frac{n}{2}+(L-1)k\leq d<\frac{n}{2}+(L-1)k+L-1 \ \text{and} \ e<\frac{n}{2}+(L-1)k-\frac{L}{2}.
\end{align*}
Note that there are $n-L$ stones, so applying inductive hypothesis (5) yields
\begin{align*}
    \text{if \ } \frac{n}{2}+(L-1)k\leq d<\frac{n}{2}+(L-1)k+\frac{L}{2}-2 \text{ and } e<\frac{n}{2}+(L-1)k+\frac{L}{2}-1, \text{ then } W(n;d,e)=1.
\end{align*}
Note that Player $1$ having more money and Player $2$ having less money does not affect the outcome of the game in this case. Comparing the two conditions yields the conclusion that Player $1$ wins the game even if Player $2$ takes $L,$ which completes another part of the induction.
\par\noindent
For $n+1$ stones, now let us examine what happens when
\begin{align*}
    \frac{n}{2}+1+(L-1)k\leq d<\frac{n}{2}+1+(L-1)k+L-1 \ \text{and} \ e\geq\frac{n}{2}+1+(L-1)k+\frac{L}{2}-1.
\end{align*}
If both players take $1$ on their first move, then it follows that $W(n;d,e)=2.$
\par\noindent
If Player $1$ now takes $L,$ Player $2$ can choose to take $L,$ and it follows that Player $2$ also wins.
\par\noindent
The edge cases in the win conditions are trivial to prove, so we are done with our induction. \end{proof}
\begin{theorem}[Staircase Theorem for $\{1,2,3\}$]
If $A=\{1,2,3\},$ we can determine who wins using the following statements.
\begin{enumerate}
    \item $n \equiv 0\pmod{4}:$
    \begin{itemize}
        \item If $d<M_1(n)+1$ and $e<M_2(n)+1,$ then $W(n;d,e)=1.$
        \item If $d<M_1(n)+1$ and $e\geq M_2(n)+1,$ then $W(n;d,e)=2.$
        \item If $M_1(n)+1\leq d<M_1(n)+2$ and $e<M_2(n)+2,$ then $W(n;d,e)=1.$
        \item If $M_1(n)+1\leq d<M_1(n)+2$ and $e\geq M_2(n)+2,$ then $W(n;d,e)=2.$
        \item If $M_1(n)+2\leq d<M_1(n)+3$ and $e<M_2(n)+3,$ then $W(n;d,e)=1.$
        \item If $M_1(n)+2\leq d<M_1(n)+3$ and $e\geq M_2(n)+3,$ then $W(n;d,e)=2.$
        \\
        \\
        $\vdots$
        \\
        \item If $M_1(n)+k\leq d<M_1(n)+k+1$ and $e<M_2(n)+k+1,$ then $W(n;d,e)=1.$
        \item If $M_1(n)+k\leq d<M_1(n)+k+1$ and $e\geq M_2(n)+k+1,$ then $W(n;d,e)=2.$
        \\
        \\
        $\vdots$
        \\
        \item If $d\geq U_1(n)-1$ and $e<U_2(n)-1,$ then $W(n;d,e)=1.$
        \item If $d\geq U_1(n)-1$ and $e\geq U_2(n)-1,$ then $W(n;d,e)=2.$
    \end{itemize}
    \item $n\equiv 1,3\pmod{4}:$
    \begin{itemize}
        \item If $d<M_1(n)+1,$ then $W(n;d,e)=2.$
        \item If $M_1(n)+1\leq d<M_1(n)+2$ and $e<M_2(n)+1,$ then $W(n;d,e)=1.$
        \item If $M_1(n)+1\leq d<M_1(n)+2$ and $e\geq M_2(n)+1,$ then $W(n;d,e)=2.$
        \item If $M_1(n)+2\leq d<M_1(n)+3$ and $e<M_2(n)+2,$ then $W(n;d,e)=1.$
        \item If $M_1(n)+2\leq d<M_1(n)+3$ and $e\geq M_2(n)+2,$ then $W(n;d,e)=2.$
        \item If $M_1(n)+3\leq d<M_1(n)+4$ and $e<M_2(n)+3,$ then $W(n;d,e)=1.$
        \item If $M_1(n)+3\leq d<M_1(n)+4$ and $e\geq M_2(n)+3,$ then $W(n;d,e)=2.$
        \\
        \\
        $\vdots$
        \\
        \item If $M_1(n)+k+1\leq d<M_1(n)+k+2$ and $e<M_2(n)+k+1,$ then $W(n;d,e)=1.$
        \item If $M_1(n)+k+1\leq d<M_1(n)+k+2$ and $e\geq M_2(n)+k+1,$ then $W(n;d,e)=2.$
        \\
        \\
        $\vdots$
        \\
        \item If $d\geq U_1(n)-1,$ then $W(n;d,e)=1.$
    \end{itemize}
    \item $n \equiv 2\pmod{4}:$
    \begin{itemize}
        \item If $d<M_1(n)+1$ and $e<M_2(n)+1,$ then $W(n;d,e)=1.$
        \item If $d<M_1(n)+1$ and $e\geq M_2(n)+1,$ then $W(n;d,e)=2.$
        \item If $M_1(n)+1\leq d<M_1(n)+2$ and $e<M_2(n)+2,$ then $W(n;d,e)=1.$
        \item If $M_1(n)+1\leq d<M_1(n)+2$ and $e\geq M_2(n)+2,$ then $W(n;d,e)=2.$
        \item If $M_1(n)+2\leq d<M_1(n)+3$ and $e<M_2(n)+3,$ then $W(n;d,e)=1.$
        \item If $M_1(n)+2\leq d<M_1(n)+3$ and $e\geq M_2(n)+3,$ then $W(n;d,e)=2.$
        \\
        \\
        $\vdots$
        \\
        \item If $M_1(n)+k\leq d<M_1(n)+k+1$ and $e<M_2(n)+k+1,$ then $W(n;d,e)=1.$
        \item If $M_1(n)+k\leq d<M_1(n)+k+1$ and $e\geq M_2(n)+k+1,$ then $W(n;d,e)=2.$
        \\
        \\
        $\vdots$
        \\
        \item If $d\geq U_1(n)-1,$ then $W(n;d,e)=1.$
    \end{itemize}
\end{enumerate}
\end{theorem}
\begin{proof} We proceed with the same method of induction outlined in the proof of $A=\{1,2\},$ but this time, we have to account for the cases when the player takes $3.$ We leave the proof of these win conditions as an exercise to the reader. \end{proof}
\begin{theorem}[Staircase Theorem for $\{1,L,L+1\},L$ Even and $\geq 4$]
If $A=\{1,L,L+1\},L$ even and $\geq 4,$ we can determine who wins using the following statements.
\begin{enumerate}
    \item $n \equiv 0,2,4,\cdots,2k,\cdots,L-2,L+2,L+4,L+6,\cdots,L+2k+2,\cdots,2L-2\pmod{2L}:$
    \begin{itemize}
        \item If $d<M_1(n)+\frac{L}{2}-1,$ then $W(n;d,e)=2.$
        \item If $M_1(n)+\frac{L}{2}-1\leq d<M_1(n)+L-1$ and $e<M_2(n)+\frac{L}{2},$ then $W(n;d,e)=1.$
        \item If $M_1(n)+\frac{L}{2}-1\leq d<M_1(n)+L-1$ and $e\geq M_2(n)+\frac{L}{2},$ then $W(n;d,e)=2.$
        \item If $M_1(n)+L-1\leq d<M_1(n)+\frac{3L}{2}-1$ and $e<M_2(n)+L,$ then $W(n;d,e)=1.$
        \item If $M_1(n)+L-1\leq d<M_1(n)+\frac{3L}{2}-1$ and $e\geq M_2(n)+L,$ then $W(n;d,e)=2.$
        \item If $M_1(n)+\frac{3L}{2}-1\leq d<M_1(n)+2L-1$ and $e<M_2(n)+\frac{3L}{2},$ then $W(n;d,e)=1.$
        \item If $M_1(n)+\frac{3L}{2}-1\leq d<M_1(n)+2L-1$ and $e\geq M_2(n)+\frac{3L}{2},$ then $W(n;d,e)=2.$
        \\
        \\
        $\vdots$
        \\
        \item If $M_1(n)+\frac{Lk}{2}-1\leq d<M_1(n)+\frac{L(k+1)}{2}-1$ and $e<M_2(n)+\frac{Lk}{2},$ then $W(n;d,e)=1.$
        \item If $M_1(n)+\frac{Lk}{2}-1\leq d<M_1(n)+\frac{L(k+1)}{2}-1$ and $e\geq M_2(n)+\frac{Lk}{2},$ then $W(n;d,e)=2.$
        \\
        \\
        $\vdots$
        \\
        \item If $d\geq U_1(n)-\frac{L}{2}$ and $e<U_2(n)-\frac{L}{2},$ then $W(n;d,e)=1.$
        \item If $d\geq U_1(n)-\frac{L}{2}$ and $e\geq U_2(n)-\frac{L}{2},$ then $W(n;d,e)=2.$
    \end{itemize}
    \item $n\equiv 1,3,5,\cdots,2k+1,\cdots,2L-1\pmod{2L}:$
    \begin{itemize}
        \item If $d<M_1(n)+\frac{L}{2}$ and $e<M_2(n)+\frac{L}{2}-1,$ then $W(n;d,e)=1.$
        \item If $d<M_1(n)+\frac{L}{2}$ and $e\geq M_2(n)+\frac{L}{2}-1,$ then $W(n;d,e)=2.$
        \item If $M_1(n)+\frac{L}{2}\leq d<M_1(n)+L$ and $e<M_2(n)+L-1,$ then $W(n;d,e)=1.$
        \item If $M_1(n)+\frac{L}{2}\leq d<M_1(n)+L$ and $e\geq M_2(n)+L-1,$ then $W(n;d,e)=2.$
        \item If $M_1(n)+L\leq d<M_1(n)+\frac{3L}{2}$ and $e<M_2(n)+\frac{3L}{2}-1,$ then $W(n;d,e)=1.$
        \item If $M_1(n)+L\leq d<M_1(n)+\frac{3L}{2}$ and $e\geq M_2(n)+\frac{3L}{2}-1,$ then $W(n;d,e)=2.$
        \\
        \\
        $\vdots$
        \\
        \item If $M_1(n)+\frac{Lk}{2}\leq d<M_1(n)+\frac{L(k+1)}{2}$ and $e<M_2(n)+\frac{L(k+1)}{2}-1,$ then $W(n;d,e)=1.$
        \item If $M_1(n)+\frac{Lk}{2}\leq d<M_1(n)+\frac{L(k+1)}{2}$ and $e\geq M_2(n)+\frac{L(k+1)}{2}-1,$ then $W(n;d,e)=2.$
        \\
        \\
        $\vdots$
        \\
        \item If $d\geq U_1(n)-\frac{L}{2},$ then $W(n;d,e)=1.$
    \end{itemize}
    \item $n\equiv L\pmod{2L}:$
    \begin{itemize}
        \item If $d<M_1(n)+\frac{L}{2}-1,$ then $W(n;d,e)=2.$
        \item If $M_1(n)+\frac{L}{2}-1\leq d<M_1(n)+L-1$ and $e<M_2(n)+\frac{L}{2},$ then $W(n;d,e)=1.$
        \item If $M_1(n)+\frac{L}{2}-1\leq d<M_1(n)+L-1$ and $e\geq M_2(n)+\frac{L}{2},$ then $W(n;d,e)=2.$
        \item If $M_1(n)+L-1\leq d<M_1(n)+\frac{3L}{2}-1$ and $e<M_2(n)+L,$ then $W(n;d,e)=1.$
        \item If $M_1(n)+L-1\leq d<M_1(n)+\frac{3L}{2}-1$ and $e\geq M_2(n)+L,$ then $W(n;d,e)=2.$
        \item If $M_1(n)+\frac{3L}{2}-1\leq d<M_1(n)+2L-1$ and $e<M_2(n)+\frac{3L}{2},$ then $W(n;d,e)=1.$
        \item If $M_1(n)+\frac{3L}{2}-1\leq d<M_1(n)+2L-1$ and $e\geq M_2(n)+\frac{3L}{2},$ then $W(n;d,e)=2.$
        \\
        \\
        $\vdots$
        \\
        \item If $M_1(n)+\frac{Lk}{2}-1\leq d<M_1(n)+\frac{L(k+1)}{2}-1$ and $e<M_2(n)+\frac{Lk}{2},$ then $W(n;d,e)=1.$
        \item If $M_1(n)+\frac{Lk}{2}-1\leq d<M_1(n)+\frac{L(k+1)}{2}-1$ and $e\geq M_2(n)+\frac{Lk}{2},$ then $W(n;d,e)=2.$
        \\
        \\
        $\vdots$
        \\
        \item If $d\geq U_1(n)-\frac{L}{2},$ then $W(n;d,e)=1.$
    \end{itemize}
\end{enumerate}
\end{theorem}
\begin{proof}We will prove all of the statements using induction on $n$ and $n-1,$ using the same argument in the proofs for $\{1,L\}.$ 
\par\noindent
Without loss of generality, let $n\equiv0 \pmod2.$ Assume that for every even integer up to $n:$ 
\begin{itemize}
    \item If $M_1(n)+\frac{Lk}{2}-1\leq d<M_1(n)+\frac{L(k+1)}{2}-1$ and $e<M_2(n)+\frac{Lk}{2},$ then $W(n;d,e)=1.$ \hfill (9)
    \item If $M_1(n)+\frac{Lk}{2}-1\leq d<M_1(n)+\frac{L(k+1)}{2}-1$ and $e\geq M_2(n)+\frac{Lk}{2},$ then $W(n;d,e)=2.$ \hfill (10)
\end{itemize}  
Moreover, assume that for every odd integer up to $n-1:$ 
\begin{itemize}
    \item If $M_1(n)+\frac{Lk}{2}\leq d<M_1(n)+\frac{L(k+1)}{2}$ and $e<M_2(n)+\frac{L(k+1)}{2}-1,$ then $W(n;d,e)=1.$ \hfill (11)
    \item If $M_1(n)+\frac{Lk}{2}\leq d<M_1(n)+\frac{L(k+1)}{2}$ and $e\geq M_2(n)+\frac{L(k+1)}{2}-1,$ then $W(n;d,e)=2.$ \hfill (12)
\end{itemize}
\par\noindent
Then we need to prove that the equivalent claims for $n+2$ and $n+1$ stones are true, respectively. 
\par\noindent
For $n+2$ stones, let us examine what happens when
\begin{align*}
    \frac{n}{2}+2+\frac{Lk}{2}-1\leq d<\frac{n}{2}+2+\frac{L(k+1)}{2}-1 \ \text{and} \ e<\frac{n}{2}+1+\frac{Lk}{2}, 
\end{align*}
where we have used our definitions for $M_1(n)$ and $M_2(n).$
\par\noindent
If both players each take $1$ on their first moves, we are back to the game where there are $n$ stones, $M_1(n)+\frac{Lk}{2}-1\leq d<M_1(n)+\frac{L(k+1)}{2}-1,$ and $e<M_2(n)+\frac{Lk}{2},$ for which we know that $W(n;d,e)=1.$
\par\noindent
If Player $1$ takes $1,$ but Player $2$ takes $L,$ both on their first moves, then 
\begin{align*}
    \frac{n}{2}+\frac{Lk}{2}\leq d<\frac{n}{2}+\frac{L(k+1)}{2} \ \text{and} \ e<\frac{n}{2}+\frac{Lk}{2}+1-L.
\end{align*}
Note that there are $n+1-L$ stones, so applying inductive hypothesis (11) yields
\begin{align*}
    \text{if \ } \frac{n}{2}+\frac{L(k-1)}{2}+1\leq d<\frac{n}{2}+\frac{Lk}{2}+1 \text{ and } e<\frac{n}{2}+\frac{Lk}{2}, \text{ then } W(n;d,e)=1.
\end{align*}
Note that in this case, Player $1$ having more money and Player $2$ having less money does not affect the outcome. Comparing the two conditions yields the conclusion that Player $1$ wins the game even if Player $2$ takes $L,$ which completes part of the induction.
\par\noindent
If Player $1$ takes $1,$ but Player $2$ takes $L+1,$ both on their first moves, then
\begin{align*}
    \frac{n}{2}+\frac{Lk}{2}\leq d<\frac{n}{2}+\frac{L(k+1)}{2} \ \text{and} \ e<\frac{n}{2}+\frac{L(k-2)}{2}.
\end{align*}
Note that there are $n-L$ stones, so applying inductive hypothesis (9) yields
\begin{align*}
    \text{if \ } \frac{n}{2}+\frac{L(k-1)}{2}-1\leq d<\frac{n}{2}+\frac{Lk}{2}-1 \text{ and } e<\frac{n}{2}+\frac{L(k-1)}{2}, \text{ then } W(n;d,e)=1.
\end{align*}
Note that in this case, Player $1$ having more money and Player $2$ having less money does not affect the outcome. Comparing the two conditions yields the conclusion that Player $1$ wins the game even if Player $2$ takes $L+1,$ which completes another part of the induction.
\par\noindent
For $n+2$ stones, now let us examine what happens when
\begin{align*}
    \frac{n}{2}+2+\frac{Lk}{2}-1\leq d<\frac{n}{2}+2+\frac{L(k+1)}{2}-1 \ \text{and} \ e\geq\frac{n}{2}+1+\frac{Lk}{2}, 
\end{align*}
If both players take $1$ on their first move, then it follows by the same reasoning as before that $W(n;d,e)=2.$
\par\noindent
By using the same reasoning, if Player $1$ now takes $L$ or $L+1,$ Player $2$ can choose to take the same amount, so Player $2$ also wins, completing another part of the induction.
\par\noindent
For $n+1$ stones, now let us examine what happens when
\begin{align*}
     \frac{n}{2}+\frac{Lk}{2}+1\leq d<\frac{n}{2}+\frac{L(k+1)}{2}+1 \ \text{and} \ e<\frac{n}{2}+\frac{L(k+1)}{2}
\end{align*}
If both players take $1$ on their first move, it follows that Player $1$ wins by the same reasoning as before. 
\par\noindent
If Player $1$ takes 1, but Player $2$ takes $L,$ both on their first moves, then 
\begin{align*}
    \frac{n}{2}+\frac{Lk}{2}\leq d<\frac{n}{2}+\frac{L(k+1)}{2} \ \text{and} \ e<\frac{n}{2}+\frac{L(k-1)}{2}.
\end{align*}
Note that there are $n-L$ stones, so applying inductive hypothesis (9) yields
\begin{align*}
    \text{if } \frac{n}{2}+\frac{L(k-1)}{2}\leq d<\frac{n}{2}+\frac{Lk}{2} \text{ and } e<\frac{n}{2}+\frac{L(k-1)}{2}, \text{ then } W(n;d,e)=1.
\end{align*}
Note that Player $1$ having more money does not affect the outcome of the game in this case. Comparing the two conditions yields the conclusion that Player $1$ wins the game even if Player $2$ takes $L,$ which completes another part of the induction.
\par\noindent
If Player $1$ takes 1, but Player $2$ takes $L+1,$ both on their first moves, then 
\begin{align*}
    \frac{n}{2}+\frac{Lk}{2}\leq d<\frac{n}{2}+\frac{L(k+1)}{2} \ \text{and} \ e<\frac{n}{2}+\frac{L(k-1)}{2}-1.
\end{align*}
Note that there are $n-L-1$ stones, so applying inductive hypothesis (11) yields
\begin{align*}
    \text{if \ } \frac{n}{2}+\frac{L(k-1)}{2}\leq d<\frac{n}{2}+\frac{Lk}{2} \text{ and } e<\frac{n}{2}+\frac{Lk}{2}-1, \text{ then }W(n;d,e)=1.
\end{align*}
Note that Player $2$ having less money does not affect the outcome of the game in this case. Comparing the two conditions yields the conclusion that Player $1$ wins the game even if Player $2$ takes $L+1,$ which completes another part of the induction.
\par\noindent
For $n+1$ stones, now let us examine what happens when
\begin{align*}
    \frac{n}{2}+\frac{Lk}{2}+1\leq d<\frac{n}{2}+\frac{L(k+1)}{2}+1 \ \text{and} \ e\geq\frac{n}{2}+\frac{L(k+1)}{2}.
\end{align*}
If both players take $1$ on their first move, then it follows that $W(n;d,e)=2.$
\par\noindent
If Player $1$ now takes $L$ or $L+1,$ Player $2$ can choose to take the same amount, and it follows that Player $2$ also wins.
\par\noindent
The ending subcases in the win conditions are trivial to prove, so we are done with our induction.
\end{proof}
\begin{theorem}[Staircase Theorem for $\{1,L,L+1\},L$ Odd]
\par\noindent
If $A=\{1,L,L+1\},L$ odd, we can determine who wins using the following statements.
\begin{enumerate}
    \item If $n\equiv 0,2,4,\cdots,2k,\cdots,L-1,3L+3,3L+5,3L+7,\cdots,3L+2k+3,\cdots,4L \pmod{4L+2}:$
    \begin{itemize}
        \item If $d<M_1(n)+\frac{L-1}{2},$ then $W(n;d,e)=2.$
        \item If $M_1(n)+\frac{L-1}{2}\leq d<M_1(n)+\frac{L-1}{2}+L$ and $e<M_2(n)+L,$ then $W(n;d,e)=1.$
        \item If $M_1(n)+\frac{L-1}{2}\leq d<M_1(n)+\frac{L-1}{2}+L$ and $e\geq M_2(n)+L,$ then $W(n;d,e)=2.$
        \item If $M_1(n)+\frac{L-1}{2}+L\leq d<M_1(n)+\frac{L-1}{2}+2L$ and $e<M_2(n)+2L,$ then $W(n;d,e)=1.$
        \item If $M_1(n)+\frac{L-1}{2}+L\leq d<M_1(n)+\frac{L-1}{2}+2L$ and $e\geq M_2(n)+2L,$ then $W(n;d,e)=2.$
        \item If $M_1(n)+\frac{L-1}{2}+2L\leq d<M_1(n)+\frac{L-1}{2}+3L$ and $e<M_2(n)+3L,$ then $W(n;d,e)=1.$
        \item If $M_1(n)+\frac{L-1}{2}+2L\leq d<M_1(n)+\frac{L-1}{2}+3L$ and $e\geq M_2(n)+3L,$ then $W(n;d,e)=2.$
        \\
        \\
        $\vdots$
        \\
        \item If $M_1(n)+\frac{L-1}{2}+Lk\leq d<M_1(n)+\frac{L-1}{2}+L(k+1)$ and $e<M_2(n)+L(k+1),$ then $W(n;d,e)=1.$
        \item If $M_1(n)+\frac{L-1}{2}+Lk\leq d<M_1(n)+\frac{L-1}{2}+L(k+1)$ and $e\geq M_2(n)+L(k+1),$ then $W(n;d,e)=2.$
        \\
        \\
        $\vdots$
        \\
        \item If $d\geq U_1(n)-\frac{L-1}{2},$ then $W(n;d,e)=1.$
    \end{itemize}
    \item If $n\equiv 1,3,5,\cdots,2k+1,\cdots,L,3L+4,3L+6,\cdots,3L+2k+4,\cdots,4L+1 \pmod{4L+2}:$
    \begin{itemize}
        \item If $d<M_1(n)+L$ and $e<M_2(n)+\frac{L-1}{2},$ then $W(n;d,e)=1.$
        \item If $d<M_1(n)+L$ and $e\geq M_2(n)+\frac{L-1}{2},$ then $W(n;d,e)=2.$
        \item If $M_1(n)+L\leq d<M_1(n)+2L$ and $e<M_2(n)+\frac{L-1}{2}+L,$ then $W(n;d,e)=1.$
        \item If $M_1(n)+L\leq d<M_1(n)+2L$ and $e\geq M_2(n)+\frac{L-1}{2}+L,$ then $W(n;d,e)=2.$
        \item If $M_1(n)+2L\leq d<M_1(n)+3L$ and $e<M_2(n)+\frac{L-1}{2}+2L,$ then $W(n;d,e)=1.$
        \item If $M_1(n)+2L\leq d<M_1(n)+3L$ and $e\geq M_2(n)+\frac{L-1}{2}+2L,$ then $W(n;d,e)=2.$
        \\
        \\
        $\vdots$
        \\
        \item If $M_1(n)+Lk\leq d<M_1(n)+L(k+1)$ and $e<M_2(n)+\frac{L-1}{2}+Lk,$ then $W(n;d,e)=1.$
        \item If $M_1(n)+Lk\leq d<M_1(n)+L(k+1)$ and $e\geq M_2(n)+\frac{L-1}{2}+Lk,$ then $W(n;d,e)=2.$
        \\
        \\
        $\vdots$
        \\
        \item If $d\geq U_1(n)-L$ and $e<U_2(n)-\frac{L-1}{2},$ then $W(n;d,e)=1.$
        \item If $d\geq U_1(n)-L$ and $e\geq U_2(n)-\frac{L-1}{2},$ then $W(n;d,e)=2.$
    \end{itemize}
    \item If $n\equiv L+1 \pmod{4L+2}:$
    \begin{itemize}
    \item If $d<M_1(n)+\frac{L-1}{2},$ then $W(n;d,e)=2.$
    \item If $M_1(n)+\frac{L-1}{2}\leq d<M_1(n)+\frac{L-1}{2}+L$ and $e<M_2(n)+L,$ then $W(n;d,e)=1.$
    \item If $M_1(n)+\frac{L-1}{2}\leq d<M_1(n)+\frac{L-1}{2}+L$ and $e\geq M_2(n)+L,$ then $W(n;d,e)=2.$
    \item If $M_1(n)+\frac{L-1}{2}+L\leq d<M_1(n)+\frac{L-1}{2}+2L$ and $e<M_2(n)+2L,$ then $W(n;d,e)=1.$
    \item If $M_1(n)+\frac{L-1}{2}+L\leq d<M_1(n)+\frac{L-1}{2}+2L$ and $e\geq M_2(n)+2L,$ then $W(n;d,e)=2.$
    \item If $M_1(n)+\frac{L-1}{2}+2L\leq d<M_1(n)+\frac{L-1}{2}+3L$ and $e<M_2(n)+3L,$ then $W(n;d,e)=1.$
    \item If $M_1(n)+\frac{L-1}{2}+2L\leq d<M_1(n)+\frac{L-1}{2}+3L$ and $e\geq M_2(n)+3L,$ then $W(n;d,e)=2.$
    \\
    \\
    $\vdots$
    \\
    \item If $M_1(n)+\frac{L-1}{2}+Lk\leq d<M_1(n)+\frac{L-1}{2}+L(k+1)$ and $e<M_2(n)+L(k+1),$ then $W(n;d,e)=1.$
    \item If $M_1(n)+\frac{L-1}{2}+Lk\leq d<M_1(n)+\frac{L-1}{2}+L(k+1)$ and $e\geq M_2(n)+L(k+1),$ then $W(n;d,e)=2.$
    \\ 
    \\
    $\vdots$
    \\
    \item If $d\geq U_1(n)-L,$ then $W(n;d,e)=1.$
    \end{itemize}
    \item If $n\equiv L+2,L+4,L+6,\cdots,L+2k+2,\cdots,3L \pmod{4L+2}:$
    \begin{itemize}
        \item If $d<M_1(n)+L$ and $e<M_2(n)+\frac{L-1}{2},$ then $W(n;d,e)=1.$
        \item If $d<M_1(n)+L$ and $e\geq M_2(n)+\frac{L-1}{2},$ then $W(n;d,e)=2.$
        \item If $M_1(n)+L\leq d<M_1(n)+2L$ and $e<M_2(n)+\frac{L-1}{2}+L,$ then $W(n;d,e)=1.$
        \item If $M_1(n)+L\leq d<M_1(n)+2L$ and $e\geq M_2(n)+\frac{L-1}{2}+L,$ then $W(n;d,e)=2.$
        \item If $M_1(n)+2L\leq d<M_1(n)+3L$ and $e<M_2(n)+\frac{L-1}{2}+2L,$ then $W(n;d,e)=1.$
        \item If $M_1(n)+2L\leq d<M_1(n)+3L$ and $e\geq M_2(n)+\frac{L-1}{2}+2L,$ then $W(n;d,e)=2.$
        \\
        \\
        $\vdots$
        \\
        \item If $M_1(n)+Lk\leq d<M_1(n)+L(k+1)$ and $e<M_2(n)+\frac{L-1}{2}+Lk,$ then $W(n;d,e)=1.$
        \item If $M_1(n)+Lk\leq d<M_1(n)+L(k+1)$ and $e\geq M_2(n)+\frac{L-1}{2}+Lk,$ then $W(n;d,e)=2.$
        \\
        \\
        $\vdots$
        \\
        \item If $d\geq U_1(n)-\frac{L-1}{2},$ then $W(n;d,e)=1.$
    \end{itemize}
    \item If $n\equiv L+3,L+5,\cdots,L+2k+1,\cdots,3L+1 \pmod{4L+2}:$
    \begin{itemize}
        \item If $d<M_1(n)+\frac{L-1}{2},$ then $W(n;d,e)=2.$
        \item If $M_1(n)+\frac{L-1}{2}\leq d<M_1(n)+\frac{L-1}{2}+L$ and $e<M_2(n)+L,$ then $W(n;d,e)=1.$
        \item If $M_1(n)+\frac{L-1}{2}\leq d<M_1(n)+\frac{L-1}{2}+L$ and $e\geq M_2(n)+L,$ then $W(n;d,e)=2.$
        \item If $M_1(n)+\frac{L-1}{2}+L\leq d<M_1(n)+\frac{L-1}{2}+2L$ and $e<M_2(n)+2L,$ then $W(n;d,e)=1.$
        \item If $M_1(n)+\frac{L-1}{2}+L\leq d<M_1(n)+\frac{L-1}{2}+2L$ and $e\geq M_2(n)+2L,$ then $W(n;d,e)=2.$
        \item If $M_1(n)+\frac{L-1}{2}+2L\leq d<M_1(n)+\frac{L-1}{2}+3L$ and $e<M_2(n)+3L,$ then $W(n;d,e)=1.$
        \item If $M_1(n)+\frac{L-1}{2}+2L\leq d<M_1(n)+\frac{L-1}{2}+3L$ and $e\geq M_2(n)+3L,$ then $W(n;d,e)=2.$
        \\
        \\
        $\vdots$
        \\
        \item If $M_1(n)+\frac{L-1}{2}+Lk\leq d<M_1(n)+\frac{L-1}{2}+L(k+1)$ and $e<M_2(n)+L(k+1),$ then $W(n;d,e)=1.$
        \item If $M_1(n)+\frac{L-1}{2}+Lk\leq d<M_1(n)+\frac{L-1}{2}+L(k+1)$ and $e\geq M_2(n)+L(k+1),$ then $W(n;d,e)=2.$
        \\
        \\
        $\vdots$
        \\
        \item If $d\geq U_1(n)-L$ and $e<U_2(n)-\frac{L-1}{2},$ then $W(n;d,e)=1.$
        \item If $d\geq U_1(n)-L$ and $e\geq U_2(n)-\frac{L-1}{2},$ then $W(n;d,e)=2.$
    \end{itemize}    
    \item If $n\equiv 3L+2 \pmod{4L+2}:$
    \begin{itemize}
        \item If $d<M_1(n)+L$ and $e<M_2(n)+\frac{L-1}{2},$ then $W(n;d,e)=1.$
        \item If $d<M_1(n)+L$ and $e\geq M_2(n)+\frac{L-1}{2},$ then $W(n;d,e)=2.$
        \item If $M_1(n)+L\leq d<M_1(n)+2L$ and $e<M_2(n)+\frac{L-1}{2}+L,$ then $W(n;d,e)=1.$
        \item If $M_1(n)+L\leq d<M_1(n)+2L$ and $e\geq M_2(n)+\frac{L-1}{2}+L,$ then $W(n;d,e)=2.$
        \item If $M_1(n)+2L\leq d<M_1(n)+3L$ and $e<M_2(n)+\frac{L-1}{2}+2L,$ then $W(n;d,e)=1.$
        \item If $M_1(n)+2L\leq d<M_1(n)+3L$ and $e\geq M_2(n)+\frac{L-1}{2}+2L,$ then $W(n;d,e)=2.$
        \\
        \\
        $\vdots$
        \\
        \item If $M_1(n)+Lk\leq d<M_1(n)+L(k+1)$ and $e<M_2(n)+\frac{L-1}{2}+Lk,$ then $W(n;d,e)=1.$
        \item If $M_1(n)+Lk\leq d<M_1(n)+L(k+1)$ and $e\geq M_2(n)+\frac{L-1}{2}+Lk,$ then $W(n;d,e)=2.$
        \\
        \\
        $\vdots$
        \\
        \item If $d\geq U_1(n)-L,$ then $W(n;d,e)=1.$
    \end{itemize}
\end{enumerate}
\end{theorem}
\begin{proof}We employ the same argument as in the proof for $\{1,L,L+1\},L$ even. The details of the proof are left as an exercise to the reader.
\end{proof}
\par This proves all of our conjectured win conditions. We proceed with schematic diagrams of the middle class win conditions in order to supplement our proofs.
\begin{example} If we plot all lattice points $(d,e)$ such that both players are middle class, and color them either green or red, for when Player $1$ or Player $2$ wins, respectively, we obtain a "staircase." 
\par Here is an example for when $n=43,$ and $A=\{1,6\}:$
\begin{center}
    \begin{asy}
    unitsize(20);
    defaultpen(fontsize(10pt));
        dot((22,22),green);
dot((22,23),green);
dot((22,24),red);
dot((22,25),red);
dot((22,26),red);
dot((22,27),red);
dot((22,28),red);
dot((22,29),red);
dot((22,30),red);
dot((22,31),red);
dot((22,32),red);
dot((22,33),red);
dot((23,22),green);
dot((23,23),green);
dot((23,24),red);
dot((23,25),red);
dot((23,26),red);
dot((23,27),red);
dot((23,28),red);
dot((23,29),red);
dot((23,30),red);
dot((23,31),red);
dot((23,32),red);
dot((23,33),red);
dot((24,22),green);
dot((24,23),green);
dot((24,24),red);
dot((24,25),red);
dot((24,26),red);
dot((24,27),red);
dot((24,28),red);
dot((24,29),red);
dot((24,30),red);
dot((24,31),red);
dot((24,32),red);
dot((24,33),red);
dot((25,22),green);
dot((25,23),green);
dot((25,24),red);
dot((25,25),red);
dot((25,26),red);
dot((25,27),red);
dot((25,28),red);
dot((25,29),red);
dot((25,30),red);
dot((25,31),red);
dot((25,32),red);
dot((25,33),red);
dot((26,22),green);
dot((26,23),green);
dot((26,24),red);
dot((26,25),red);
dot((26,26),red);
dot((26,27),red);
dot((26,28),red);
dot((26,29),red);
dot((26,30),red);
dot((26,31),red);
dot((26,32),red);
dot((26,33),red);
dot((27,22),green);
dot((27,23),green);
dot((27,24),green);
dot((27,25),green);
dot((27,26),green);
dot((27,27),green);
dot((27,28),green);
dot((27,29),red);
dot((27,30),red);
dot((27,31),red);
dot((27,32),red);
dot((27,33),red);
dot((28,22),green);
dot((28,23),green);
dot((28,24),green);
dot((28,25),green);
dot((28,26),green);
dot((28,27),green);
dot((28,28),green);
dot((28,29),red);
dot((28,30),red);
dot((28,31),red);
dot((28,32),red);
dot((28,33),red);
dot((29,22),green);
dot((29,23),green);
dot((29,24),green);
dot((29,25),green);
dot((29,26),green);
dot((29,27),green);
dot((29,28),green);
dot((29,29),red);
dot((29,30),red);
dot((29,31),red);
dot((29,32),red);
dot((29,33),red);
dot((30,22),green);
dot((30,23),green);
dot((30,24),green);
dot((30,25),green);
dot((30,26),green);
dot((30,27),green);
dot((30,28),green);
dot((30,29),red);
dot((30,30),red);
dot((30,31),red);
dot((30,32),red);
dot((30,33),red);
dot((31,22),green);
dot((31,23),green);
dot((31,24),green);
dot((31,25),green);
dot((31,26),green);
dot((31,27),green);
dot((31,28),green);
dot((31,29),red);
dot((31,30),red);
dot((31,31),red);
dot((31,32),red);
dot((31,33),red);
dot((32,22),green); 
dot((32,23),green);
dot((32,24),green);
dot((32,25),green);
dot((32,26),green);
dot((32,27),green);
dot((32,28),green);
dot((32,29),green);
dot((32,30),green);
dot((32,31),green);
dot((32,32),green);
dot((32,33),green);
dot((33,22),green);
dot((33,23),green);
dot((33,24),green);
dot((33,25),green);
dot((33,26),green);
dot((33,27),green);
dot((33,28),green);
dot((33,29),green);
dot((33,30),green);
dot((33,31),green);
dot((33,32),green);
dot((33,33),green);
dot((34,22),green);
dot((34,23),green);
dot((34,24),green);
dot((34,25),green);
dot((34,26),green);
dot((34,27),green);
dot((34,28),green);
dot((34,29),green);
dot((34,30),green);
dot((34,31),green);
dot((34,32),green);
dot((34,33),green);
dot((35,22),green);
dot((35,23),green);
dot((35,24),green);
dot((35,25),green);
dot((35,26),green);
dot((35,27),green);
dot((35,28),green);
dot((35,29),green);
dot((35,30),green);
dot((35,31),green);
dot((35,32),green);
dot((35,33),green);
dot((36,22),green);
dot((36,23),green);
dot((36,24),green);
dot((36,25),green);
dot((36,26),green);
dot((36,27),green);
dot((36,28),green);
dot((36,29),green);
dot((36,30),green);
dot((36,31),green);
dot((36,32),green);
dot((36,33),green);
label("22",(22,21));
label("23",(23,21));
label("24",(24,21));
label("25",(25,21));
label("26",(26,21));
label("27",(27,21));
label("28",(28,21));
label("29",(29,21));
label("30",(30,21));
label("31",(31,21));
label("32",(32,21));
label("33",(33,21));
label("34",(34,21));
label("35",(35,21));
label("36",(36,21));
label("$d$",(37,21));
label("22",(21,22));
label("23",(21,23));
label("24",(21,24));
label("25",(21,25));
label("26",(21,26));
label("27",(21,27));
label("28",(21,28));
label("29",(21,29));
label("30",(21,30));
label("31",(21,31));
label("32",(21,32));
label("33",(21,33));
label("$e$",(21,34));
    \end{asy}
\end{center}
From our win conditions for $\{1,L\},L$ even, we can verify that each "step" on the staircase is $L-1$ units long and $L-1$ units tall. In this case, observe that each "step" is $5$ units long and $5$ units tall.
\par
Here is an example for when $n=42,$ and $A=\{1,3,4\}:$
\begin{center}
    \begin{asy}
    unitsize(20);
    defaultpen(fontsize(10pt));
    dot((22,21),red);
dot((22,22),red);
dot((22,23),red);
dot((22,24),red);
dot((22,25),red);
dot((22,26),red);
dot((22,27),red);
dot((22,28),red);
dot((22,29),red);
dot((23,21),green);
dot((23,22),green);
dot((23,23),green);
dot((23,24),red);
dot((23,25),red);
dot((23,26),red);
dot((23,27),red);
dot((23,28),red);
dot((23,29),red);
dot((24,21),green);
dot((24,22),green);
dot((24,23),green);
dot((24,24),red);
dot((24,25),red);
dot((24,26),red);
dot((24,27),red);
dot((24,28),red);
dot((24,29),red);
dot((25,21),green);
dot((25,22),green);
dot((25,23),green);
dot((25,24),red);
dot((25,25),red);
dot((25,26),red);
dot((25,27),red);
dot((25,28),red);
dot((25,29),red);
dot((26,21),green);
dot((26,22),green);
dot((26,23),green);
dot((26,24),green);
dot((26,25),green);
dot((26,26),green);
dot((26,27),red);
dot((26,28),red);
dot((26,29),red);
dot((27,21),green);
dot((27,22),green);
dot((27,23),green);
dot((27,24),green);
dot((27,25),green);
dot((27,26),green);
dot((27,27),red);
dot((27,28),red);
dot((27,29),red);
dot((28,21),green);
dot((28,22),green);
dot((28,23),green);
dot((28,24),green);
dot((28,25),green);
dot((28,26),green);
dot((28,27),red);
dot((28,28),red);
dot((28,29),red);
dot((29,21),green);
dot((29,22),green);
dot((29,23),green);
dot((29,24),green);
dot((29,25),green);
dot((29,26),green);
dot((29,27),green);
dot((29,28),green);
dot((29,29),green);
label("22",(22,20));
label("23",(23,20));
label("24",(24,20));
label("25",(25,20));
label("26",(26,20));
label("27",(27,20));
label("28",(28,20));
label("29",(29,20));
label("$d$",(30,20));
label("21",(21,21));
label("22",(21,22));
label("23",(21,23));
label("24",(21,24));
label("25",(21,25));
label("26",(21,26));
label("27",(21,27));
label("28",(21,28));
label("29",(21,29));
label("$e$",(21,30));
    \end{asy}
\end{center}
From our win conditions for $\{1,L,L+1\},L$ odd, we can verify that each "step" on the staircase is $L$ units long and $L$ units tall. In this case, observe that each "step" is $3$ units long and $3$ units tall.
\par
Here is an example for when $n=48,$ and $A=\{1,4,5\}:$
\begin{center}
    \begin{asy}
    unitsize(20);
    defaultpen(fontsize(10pt));
    dot((25,24),red);
dot((25,25),red);
dot((25,26),red);
dot((25,27),red);
dot((25,28),red);
dot((25,29),red);
dot((25,30),red);
dot((25,31),red);
dot((25,32),red);
dot((25,33),red);
dot((25,34),red);
dot((25,35),red);
dot((26,24),green);
dot((26,25),green);
dot((26,26),red);
dot((26,27),red);
dot((26,28),red);
dot((26,29),red);
dot((26,30),red);
dot((26,31),red);
dot((26,32),red);
dot((26,33),red);
dot((26,34),red);
dot((26,35),red);
dot((27,24),green);
dot((27,25),green);
dot((27,26),red);
dot((27,27),red);
dot((27,28),red);
dot((27,29),red);
dot((27,30),red);
dot((27,31),red);
dot((27,32),red);
dot((27,33),red);
dot((27,34),red);
dot((27,35),red);
dot((28,24),green);
dot((28,25),green);
dot((28,26),green);
dot((28,27),green);
dot((28,28),red);
dot((28,29),red);
dot((28,30),red);
dot((28,31),red);
dot((28,32),red);
dot((28,33),red);
dot((28,34),red);
dot((28,35),red);
dot((29,24),green);
dot((29,25),green);
dot((29,26),green);
dot((29,27),green);
dot((29,28),red);
dot((29,29),red);
dot((29,30),red);
dot((29,31),red);
dot((29,32),red);
dot((29,33),red);
dot((29,34),red);
dot((29,35),red);
dot((30,24),green);
dot((30,25),green);
dot((30,26),green);
dot((30,27),green);
dot((30,28),green);
dot((30,29),green);
dot((30,30),red);
dot((30,31),red);
dot((30,32),red);
dot((30,33),red);
dot((30,34),red);
dot((30,35),red);
dot((31,24),green);
dot((31,25),green);
dot((31,26),green);
dot((31,27),green);
dot((31,28),green);
dot((31,29),green);
dot((31,30),red);
dot((31,31),red);
dot((31,32),red);
dot((31,33),red);
dot((31,34),red);
dot((31,35),red);
dot((32,24),green);
dot((32,25),green);
dot((32,26),green);
dot((32,27),green);
dot((32,28),green);
dot((32,29),green);
dot((32,30),green);
dot((32,31),green);
dot((32,32),red);
dot((32,33),red);
dot((32,34),red);
dot((32,35),red);
dot((33,24),green);
dot((33,25),green);
dot((33,26),green);
dot((33,27),green);
dot((33,28),green);
dot((33,29),green);
dot((33,30),green);
dot((33,31),green);
dot((33,32),red);
dot((33,33),red);
dot((33,34),red);
dot((33,35),red);
dot((34,24),green);
dot((34,25),green);
dot((34,26),green);
dot((34,27),green);
dot((34,28),green);
dot((34,29),green);
dot((34,30),green);
dot((34,31),green);
dot((34,32),green);
dot((34,33),green);
dot((34,34),red);
dot((34,35),red);
dot((35,24),green);
dot((35,25),green);
dot((35,26),green);
dot((35,27),green);
dot((35,28),green);
dot((35,29),green);
dot((35,30),green);
dot((35,31),green);
dot((35,32),green);
dot((35,33),green);
dot((35,34),red);
dot((35,35),red);
label("25",(25,23));
label("26",(26,23));
label("27",(27,23));
label("28",(28,23));
label("29",(29,23));
label("30",(30,23));
label("31",(31,23));
label("32",(32,23));
label("33",(33,23));
label("34",(34,23));
label("35",(35,23));
label("$d$",(36,23));
label("24",(24,24));
label("25",(24,25));
label("26",(24,26));
label("27",(24,27));
label("28",(24,28));
label("29",(24,29));
label("30",(24,30));
label("31",(24,31));
label("32",(24,32));
label("33",(24,33));
label("34",(24,34));
label("35",(24,35));
label("$e$",(24,36));
    \end{asy}
\end{center}
From our win conditions for $\{1,L,L+1\},L$ even, we can verify that each "step" on the staircase is $\frac{L}{2}$ units long and $\frac{L}{2}$ units tall. In this case, observe that each "step" is $2$ units long and $2$ units tall.
\end{example}

%\bibliographystyle{abbrv}
%\bibliography{bibfile}

\begin{thebibliography}{1}

\bibitem{combgamesbib}
A.~Fraenkel.
\newblock Selected bibliography on combinatorial games.
\newblock {\em Electronic Journal of Combinatorics}, DS2, 1994.
\newblock \url{http://www.combinatorics.org}, Dynamic Survey.

\bibitem{partizan}
A.~Fraenkel and A.~Kotzig.
\newblock Partizan octal games; partizan subtraction games.
\newblock {\em International Journal of Game Theory}, 16:145--154, 1987.
\newblock {\url{http://link.springer.com/article/10.1007%2FBF01780638}}.

\bibitem{nwcold}
W.~Gasarch, J.~Purtilo, and D.~Ulrich.
\newblock Nim with cash, 2015.
\newblock \url{https://arxiv.org/pdf/1511.04035.pdf}.

\bibitem{grundy}
P.~Grundy.
\newblock Mathematics and games.
\newblock {\em Eueka}, 2:6--8, 1939.

\bibitem{NIMarb}
A.~Holshouser and H.~Reiter.
\newblock One pile nim with arbitrary move function.
\newblock {\em Electronic Journal of Combinatorics}, 10, 2003.
\newblock \url{http://www.combinatorics.org}, Dynamic Survey.

\bibitem{cookie}
J.~Propp.
\newblock On the cookie game.
\newblock {\em International Journal of Game Theory}, 20:313--324, 1992.
\newblock \url{http://faculty.uml.edu/jpropp/articles.html}.

\bibitem{threenim}
J.~Propp.
\newblock Three-person impartial games.
\newblock {\em TCS}, 223:263--278, 2000.
\newblock \url{http://faculty.uml.edu/jpropp/articles.html}.

\bibitem{sprag}
R.~Sprague.
\newblock Uber mathematische kampsfpiele.
\newblock {\em Tohoku Math}, 41:438--444, 1936.

\end{thebibliography}

\end{document}